\documentclass[11pt]{article}

\usepackage{amsmath, amsthm,amssymb,mathtools}  
\usepackage{fullpage}
\usepackage{authblk}
\usepackage{titlesec}
\titlespacing*{\paragraph}{0pt}{1ex plus 1ex minus .2ex}{1em}
\usepackage{graphicx}
\usepackage{epstopdf}
\usepackage{seqsplit}
\usepackage{url}

\usepackage{enumerate}

\newtheorem{theorem}{Theorem}[section]
\newtheorem{lemma}[theorem]{Lemma} 
 
\newtheorem{corollary}[theorem]{Corollary} 
\newtheorem{conjecture}[theorem]{Conjecture}

\newcommand{\cC}{{\mathcal C}}

\newcommand{\cE}{{\mathcal E}}
\newcommand{\cL}{{\mathcal L}}
\newcommand{\cP}{{\mathcal P}}
\newcommand{\cS}{{\mathcal S}}
\newcommand{\cT}{{\mathcal T}}
\newcommand{\cM}{{\mathcal M}}

\title{On the Maximum Agreement Subtree Conjecture for Balanced Trees}
\author[1]{Magnus Bordewich}
\author[2]{Simone Linz}
\author[3,5]{Megan Owen}
\author[4,5]{Katherine St.~John}
\author[6]{Charles Semple}
\author[7]{Kristina Wicke}
\affil[1]{Department of Computer Science, Durham University, United Kingdom}
\affil[2]{School of Computer Science, University of Auckland, Auckland, New Zealand}
\affil[3]{Department of Mathematics, Lehman College, City University of New York, United States}
\affil[4]{Department of Computer Science, Hunter College, City University of New York, United States}
\affil[5]{Division of Invertebrate Zoology, American Museum of Natural History, New York, United States}
\affil[6]{School of Mathematics and Statistics, University of Canterbury, Christchurch, New Zealand}
\affil[7]{Institute of Mathematics and Computer Science, University of Greifswald, Greifswald, Germany}

\begin{document}
\maketitle

\begin{abstract}
We give a counterexample to the conjecture of Martin and Thatte that two balanced rooted binary leaf-labelled trees on $n$ leaves have a maximum agreement subtree (MAST) of size at least $n^{\frac{1}{2}}$. In particular, we show that for any $c>0$, there exist two balanced rooted binary leaf-labelled trees on $n$ leaves such that any MAST for these two trees has size less than $c n^{\frac{1}{2}}$. We also improve the lower bound of the size of such a MAST to $n^{\frac{1}{6}}$. 
\end{abstract}


\section{Introduction}

Leaf-labelled trees, also called phylogenetic trees, are used to represent inferred evolutionary histories of sets of species \cite{sempleSteelBook}. Due to different data sets  and methods that are used to infer  such histories, even phylogenetic trees that have been reconstructed for the same set of species often differ. To quantify the dissimilarities between trees, several tree metrics are commonly used that compute distances between two phylogenetic trees (for a recent review, see~\cite{st2017shape}). Moreover, to summarise the information that two or more phylogenetic trees have in common, maximum agreement subtrees (MASTs) have become a popular tool. Historically, the concept of a MAST was introduced by Finden and Gordon \cite{finden1985} as a way of measuring  similarities among an arbitrary number of phylogenetic trees. Intuitively, an agreement subtree for a collection $\cP$ of phylogenetic trees is a leaf-labelled tree $\cM$ that can be embedded in each tree in $\cP$. If $\cM$ has the maximum number of leaves over all agreement subtrees for $\cP$, then $\cM$ is a maximum agreement subtree  (formal definitions are given below).

Various aspects of MASTs have been well studied over the years. While the problem of computing a MAST for at least three phylogenetic trees is NP-hard in general~\cite{amir1997maximum}, several polynomial-time  algorithms have been developed to compute a MAST for two binary phylogenetic trees \cite{cole2000n,goddard1994agreement,steel1993kaikoura}. 
In terms of a lower bound on the size of a MAST, Martin and Thatte \cite{martin2013maximum} have shown that two unrooted binary phylogenetic trees on $n$ leaves have a MAST on $\Omega(\sqrt{\log n})$ leaves. Their result improves on a lower bound of $\Omega(\log\log n)$ leaves that was  previously established by Steel and Sz\'ekely \cite{steel2009improved}.
Markin \cite{markin2018extremal} has recently closed the gap between lower and upper bound asymptotics by showing that the minimum MAST for two unrooted binary phylogenetic trees has  $\Theta(\sqrt{\log n})$ leaves.  Turning to phylogenetic trees of a particular shape, Martin and Thatte
\cite{martin2013maximum} have  investigated the size of a MAST for two balanced rooted binary phylogenetic trees on $n=2^m$ leaves. In particular, they have shown that a MAST for two such trees has at least $2^{\beta m}$ leaves, where $\beta \sim 0.149$ and, subsequently, conjectured that  two balanced rooted binary phylogenetic trees on $n$ leaves have a MAST with at least $n^{\frac 1 2}$ leaves. The main result of this paper, disproves their conjecture for infinitely many values of $n$ via a family of counterexamples. 

For completeness, it is also worth mentioning that 
lower and upper bounds on the expected size of a MAST for two phylogenetic trees that are generated under the uniform or Yule-Harding distribution have been established in \cite{bernstein2015bounds,bryant2003size,misra2019bounds}. Interestingly, in the context of this paper, Misra and Sullivant~\cite{misra2019bounds} have shown that the expected size of a MAST for two balanced rooted binary phylogenetic trees on $n$ leaves is $\Theta(n^{\frac 1 2})$ under the uniform distribution.

To formally state the main results, we require some terminology. A {\it rooted phylogenetic $X$-tree} $\cT$ is a rooted tree with leaf set $X$ and with no degree-two vertices, except for the root which has degree at least two. For technical reasons, if $|X|=1$, we additionally allow $\cT$ to consist of the single vertex in $X$, in which case, this vertex is the root as well as the leaf of $\cT$. The {\it size} of $\cT$ is $|X|$. If $|X|=1$ or $\cT$ has the property that the root has degree two and all other interior vertices have degree three, then $\cT$ is {\it binary}. Furthermore, the {\it height} of $\cT$ is the number of edges on the longest path from the root to a leaf. If $\cT$ is binary, we say $\cT$ is {\em balanced} if the size of $\cT$ is $2^m$ for some non-negative integer $m$ and the height of $\cT$ is $m$.

Let $\cT$ be a rooted binary phylogenetic $X$-tree, and let $Y$ be a subset of $X$. Then, the {\it restriction of $\cT$ to $Y$}, denoted by $\cT|Y$, is the rooted phylogenetic $Y$-tree obtained from the minimal subtree of $\cT$ that connects all leaves in $Y$ by suppressing all non-root degree-two vertices. Now let $\cS$ be a rooted binary phylogenetic $X'$-tree. If $Y$ is a subset of $X\cap X'$ such that $\cS|Y$ and $\cT|Y$ are isomorphic, we call $\cS|Y$ an {\it agreement subtree} of $\cS$ and $\cT$. If, amongst all agreement subtrees of $\cS$ and $\cT$, the restriction $\cS|Y$ is of maximum size, then it is called a {\it maximum agreement subtree} (MAST) of $\cS$ and $\cT$, and $|Y|$ is denoted by $\textrm{mast}(\cS, \cT)$.

The first main result of the paper, Theorem~\ref{mainresult}, shows that two balanced rooted binary phylogenetic trees on $n$ leaves do not necessarily have a MAST of size at least $n^{\frac{1}{2}}$.

\begin{theorem}\label{mainresult}
For any $c>0$, there exist balanced rooted binary phylogenetic $X$-trees $\cS$ and $\cT$, where $n=|X|$, such that ${\rm mast}(\cS, \cT) < cn^{\frac{1}{2}}$.
\end{theorem}

\noindent Theorem~\ref{mainresult} disproves the aforementioned conjecture by Martin and Thatte~\cite[Conjecture 20]{martin2013maximum} that we state next.

\begin{conjecture}[Martin and Thatte~\cite{martin2013maximum}]\label{conjecture}
If $\cS$ and $\cT$ are two balanced rooted binary phylogenetic $X$-trees, where $n=|X|$, then ${\rm mast}(\cS, \cT) \geq n^{\frac 1 2}$.
\end{conjecture}

{The second main result, Theorem~\ref{thm:lowerbound}, slightly improves the lower bound on the size of a MAST for a pair of balanced rooted binary phylogenetic trees given in \cite{martin2013maximum} from $n^\beta$, where $\beta\sim 0.149$, to $n^{0.17}$.}
\begin{theorem}\label{thm:lowerbound}
{If $\cS$ and $\cT$ are two balanced rooted binary phylogenetic $X$-trees, where $n=|X|$, then ${\rm mast}(\cS, \cT)\ge n^{0.17}>n^{\frac{1}{6}}$.}
\end{theorem}

This paper is organised as follows. Section~\ref{sec:notation} details the notation and definitions used throughout, while Section~\ref{sec:outline} outlines the counterexample construction for when $n=2048$. Section~\ref{sec:proofs} establishes the correctness of the construction and, more particularly, proves Theorem~\ref{mainresult} via a sequence of lemmas. The last section, Section~\ref{sec:lowerBounds} gives the proof of the improved lower bound on the size of a MAST (Theorem~\ref{thm:lowerbound}).

\section{Notation and Preliminaries}
\label{sec:notation}

In addition to the terminology given in the introduction, this section provides notation and terminology that is used in the remaining sections. Throughout the paper, $X$ denotes a non-empty finite set and all logarithms are base $2$.

\paragraph{Trees and subtrees.} Since all phylogenetic trees in this paper are rooted and binary, we will refer to such a tree as simply a {\it phylogenetic tree}. Let $\cT$ be a phylogenetic $X$-tree. We call $X$ the {\it label set} of $\cT$ and frequently denote it by $\cL(\cT)$. A  subtree of $\cT$ is {\it pendant} if it can be detached from $\cT$ by deleting a single edge. 
Observe that if $\cT$ is balanced, then the size of $\cT$ and each pendant subtree of $\cT$ is a power of two. A balanced phylogenetic tree on $n$ leaves has height $\log n$. As an example, Figure~\ref{fig:8caterpillars} shows such a tree on $8$ leaves.

Of course, a phylogenetic tree is simply a certain type of rooted tree whose leaves are labelled. Thus, if we omit the adjective ``phylogenetic'' and say, for example, a balanced tree, we mean a balanced rooted binary tree whose leaf set is unlabelled. This occurs, in particular, in some of the proofs where we describe processes that start with an (unlabelled) tree and assign labels to its leaves so that the resulting tree is a phylogenetic tree. 

\paragraph{Caterpillars.} Let $\cT$ be a phylogenetic $X$-tree with $n=|X|$. We call $\cT$ an {\it $n$-caterpillar} or, simply, a {\em caterpillar} if $n=1$, or $n\ge 2$ and we can order its leaf set $X$, say $l_1, l_2, \ldots, l_n$, so that $l_1$ and $l_2$ have the same parent and, for all $i\in \{2, 3, \ldots, n-1\}$, we have that $(p_{i+1}, p_i)$ is an edge in $\cT$, where $p_{i+1}$ and $p_i$ are the parents of $l_{i+1}$ and $l_i$, respectively. We denote such a caterpillar $\cT$ by $(l_1, l_2, \ldots, l_n)$ or, equivalently, $(l_2, l_1, l_3, \ldots, l_n)$, where $l_1$ and $l_2$ have been interchanged. For a  caterpillar $\cC=(l_1,l_2,\ldots,l_n)$ and a leaf $l'$ with $l'\notin\{l_1,l_2,\ldots,l_n\}$, the caterpillar $(l_1,l_2,\ldots,l_n,l')$ is denoted by $\cC||l'$. If $\cC=(l_1,l_2,\ldots,l_n)$ and $\cC'=(l_1',l_2',\ldots,l_{n}')$ are two caterpillars  on the same set of leaves such that $l_i=l'_{n-i+1}$ for all $i\in\{1,2,\ldots, n\}$, we say that $\cC$ and $\cC'$ are a pair of {\it anti-caterpillars}. For example, if $\cC=(l_1, l_2, \ldots, l_8)$ and $\cC'=(l_8, l_7, \ldots, l_1)$, then $\cC$ and $\cC'$ are a pair of anti-caterpillars. Lastly, a caterpillar on $3$ leaves, say $(a,b,c)$, is called a {\it triple} and denoted by $ab|c$ or, equivalently, $ba|c$.

\begin{figure}[t]
    \centering
    \includegraphics[width=0.45\textwidth]{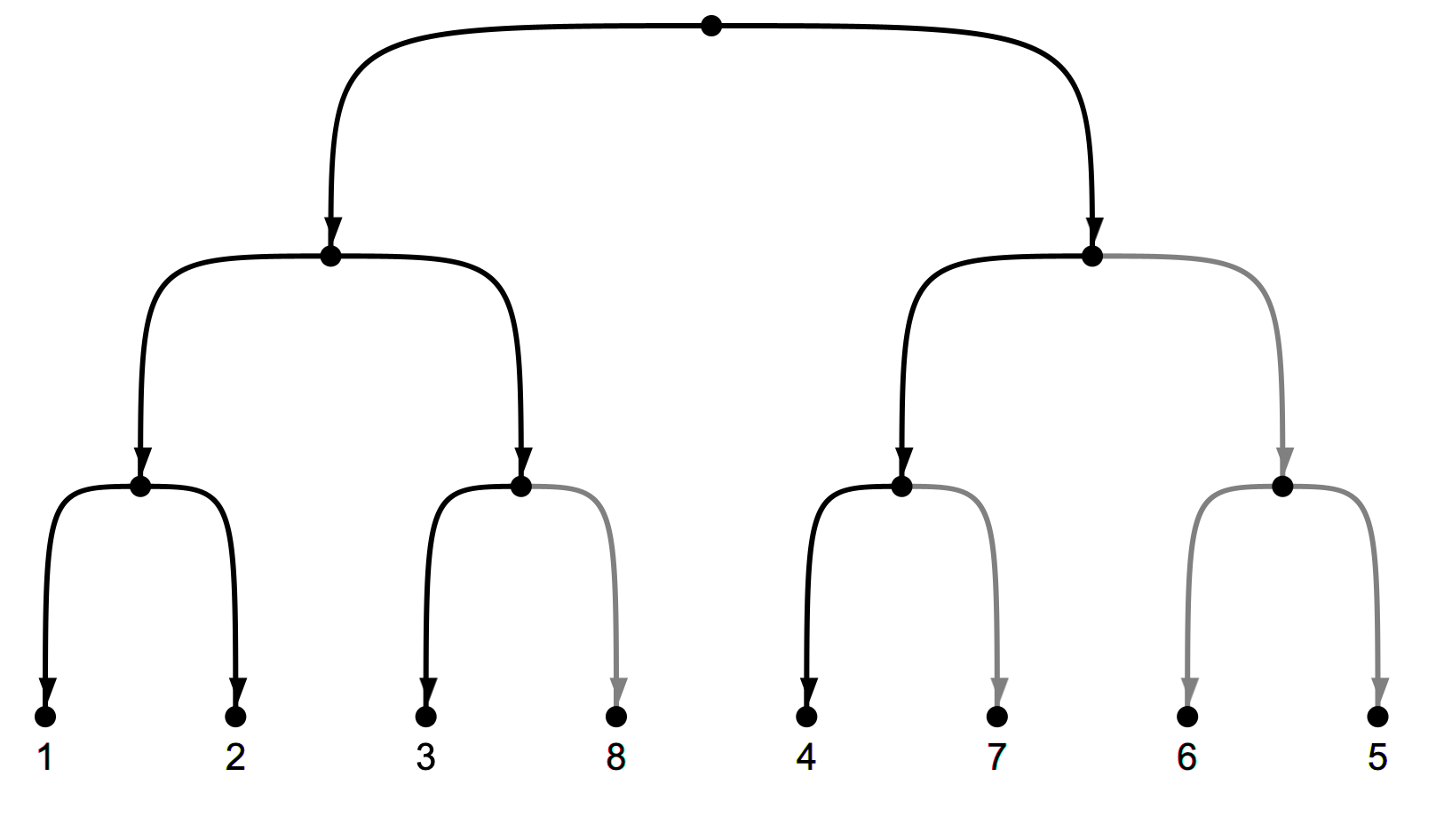}
    \caption{A balanced phylogenetic tree on $8$ leaves. The black edges delineate the embedding of the restricted subtree on leaves $\{1,2,3,4\}$ which is the caterpillar $(1,2,3,4)$. The restricted subtree on the remaining four leaves also form a caterpillar, namely $(5,6,7,8)$.}
    \label{fig:8caterpillars}
\end{figure}

\paragraph{Embedding trees.} Let $\cT$ be a phylogenetic $X$-tree, and let $\cT'$ be a phylogenetic $Y$-tree such that $Y\subseteq X$. We say that $\cT$ {\it embeds} $\cT'$ if $\cT|\cL(\cT')$ is isomorphic to $\cT'$, in which case we refer to the minimal rooted subtree of $\cT$ that connects all leaves in $Y$ as an {\it embedding} $\cE$ of $\cT'$ in $\cT$. To illustrate, the subtree formed by the black edges in the balanced phylogenetic tree in Figure~\ref{fig:8caterpillars} is an embedding of the caterpillar $(1, 2, 3, 4)$. Analogous to the label set of a phylogenetic tree, we use $\cL(\cE)$ to refer to the label set of $\cE$. More generally, let $\cP=\{\cS_1,\cS_2,\ldots,\cS_k\}$ be a set of phylogenetic trees such that, for each $i\in\{1,2,\ldots,k\}$, we have $\cL(\cS_i)\subseteq X$. Then the union $\cL(S_1)\cup \cL(\cS_2)\cup \cdots \cup \cL(S_k)$, denoted by $\cL(\cP)$, is the {\em label} set of $\cP$. If $\cT$ embeds each tree in $\cP$, we say $\cT$ {\it embeds} $\cP$, in which case, an {\em embedding}, say $\cE$, of $\cP$ in $\cT$ refers to an  embedding of each tree in $\cP$ in $\cT$ and the {\em label} set of $\cE$, denoted $\cL(\cE)$, is $\cL(\cP)$. Furthermore, if $\{\cL(\cS_1), \cL(\cS_2), \ldots, \cL(\cS_k)\}$ is a partition of $X$, then $\cT$ {\it perfectly embeds} $\cP$. Figure~\ref{fig:8caterpillars} shows a balanced phylogenetic tree on $8$~leaves that perfectly embeds the two $4$-caterpillars $(1,2,3,4)$ and $(5,6,7,8)$.

\paragraph{Maximum agreement subtrees.} Let $\cS$ and $\cT$ be two phylogenetic trees, and let $\cS_L, \cS_R$ and $\cT_L, \cT_R$ denote the two maximal pendant subtrees of $\cS$ and $\cT$, respectively. Then, $\textrm{mast}(\cS, \cT)$ can be computed recursively~\cite{steel1993kaikoura} as follows:
\begin{align} \label{eqn:mast}
    \textrm{mast}(\cS, \cT) = \max \big\{&\textrm{mast}(\cS_L,\cT_L) + \textrm{mast}(\cS_R, \cT_R), \,
    \textrm{mast}(\cS_L, \cT_R) + \textrm{mast}(\cS_R,\cT_L), \nonumber\\
    &\textrm{mast}(\cS,\cT_L), \, \textrm{mast}(\cS, \cT_R), \,
    \textrm{mast}(\cS_L, \cT), \, \textrm{mast}(\cS_R, \cT) \big\}.
\end{align}
Furthermore, a MAST for a pair of phylogenetic trees need not be unique. As an example, consider the two trees $\cS$ and $\cT$ on $16$ leaves depicted in Figure~\ref{fig:S16}. Any MAST of $\cS$ and $\cT$ has size $4$. Two such MASTs for $\cS$ and $\cT$ are $\cM_1=$\texttt{((6,7),(10,11))} and $\cM_2 =$\texttt{((5,6),(10,12))}.

\section{The Counterexample Construction}
\label{sec:outline}

{In this section, we give an explicit construction of a counterexample to Conjecture~\ref{conjecture}. Throughout the construction, we will refer to several lemmas and corollaries that are established in Section~\ref{sec:proofs}. We note that the results in Section~\ref{sec:proofs} are more general than what we need for the counterexample that is presented in this section. In particular, Theorem~\ref{t:construction} shows that, for any positive integer $k$, there exists a pair of balanced phylogenetic trees on $2^{2^{k+1}-k-2}$ leaves such that the size of a MAST for these two trees is $2^{2^k-k}$. Thus, for $k=3$, this theorem states that there exist two balanced trees $\cS$ and $\cT$ on $2^{2^{3+1}-3-2}=2^{11}=2048$ leaves that can be bijectively labelled with the elements in $\{1,2,\ldots,2048\}$ such that the size of a MAST for the resulting two balanced phylogenetic trees has size $2^{2^3-3}=32<2048^{\frac 1 2}$.}

\begin{figure}
    \centering
    \begin{minipage}{0.4\textwidth}
        \centering
        \includegraphics[width=0.5\textwidth]{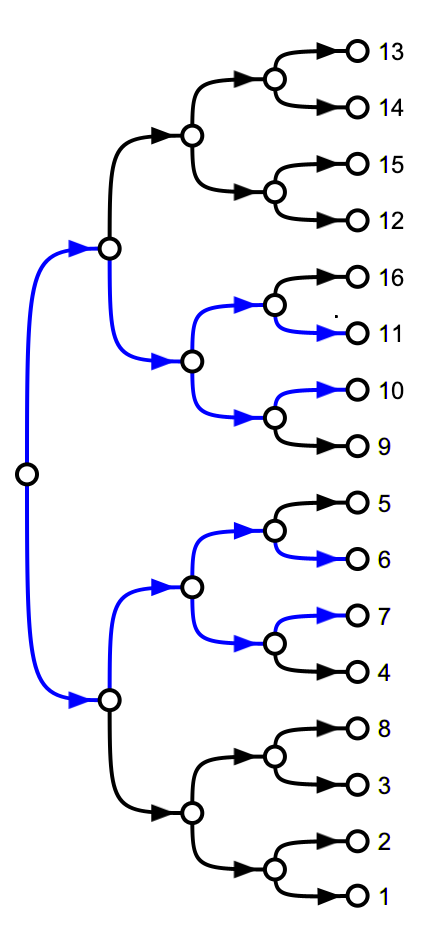}

        $\cS$
    \end{minipage}
    \begin{minipage}{0.4\textwidth}
        \centering
        \includegraphics[width=0.5\textwidth]{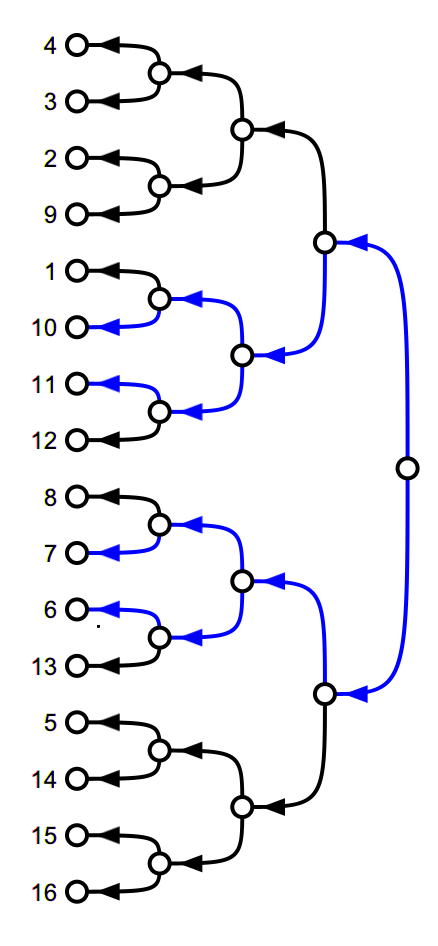}
        
        $\cT$
    \end{minipage}
    \caption{Two phylogenetic trees $\cS$ and $\cT$ on $16$ leaves, drawn with PhyloSketch~\cite{phyloSketch}.  An embedding of the MAST $\cM_1 =$\texttt{((6,7),(10,11))} for $\cS$ and $\cT$ is shown in blue. 
    }
    \label{fig:S16}    
    \label{fig:T16}
\end{figure}

\begin{figure}[ht!]
    \centering
    \begin{minipage}{0.45\textwidth}
        \centering
        \includegraphics[width=0.77\textwidth]{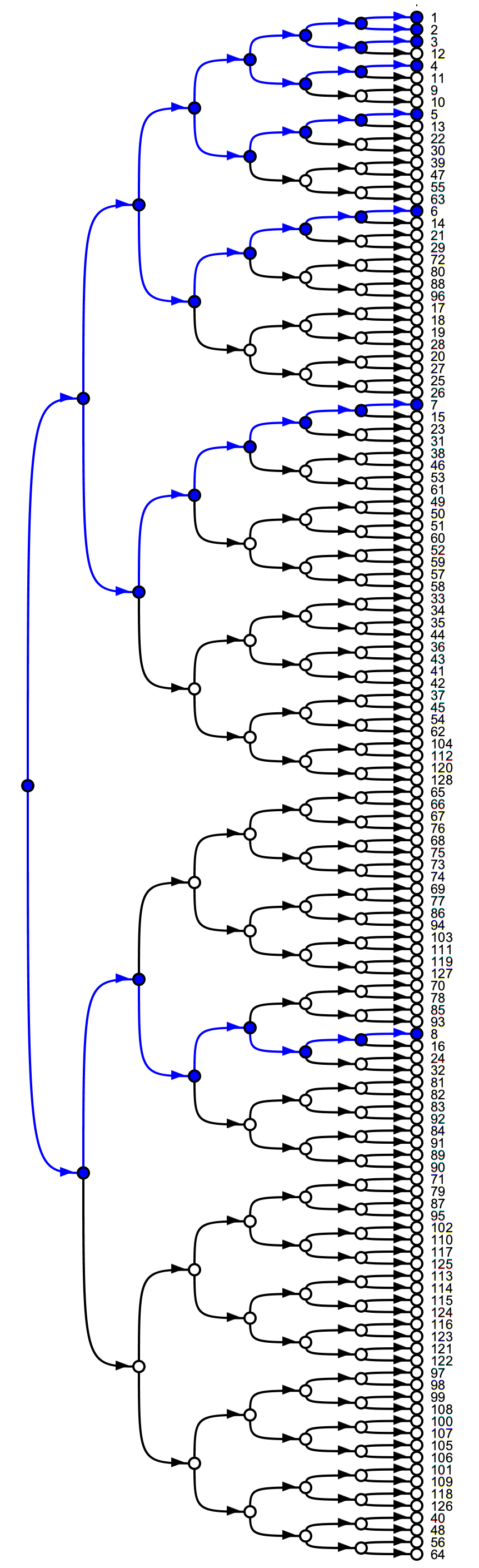}
    \end{minipage}\hfill
    \begin{minipage}{0.45\textwidth}
        \centering
        \includegraphics[width=0.77\textwidth]{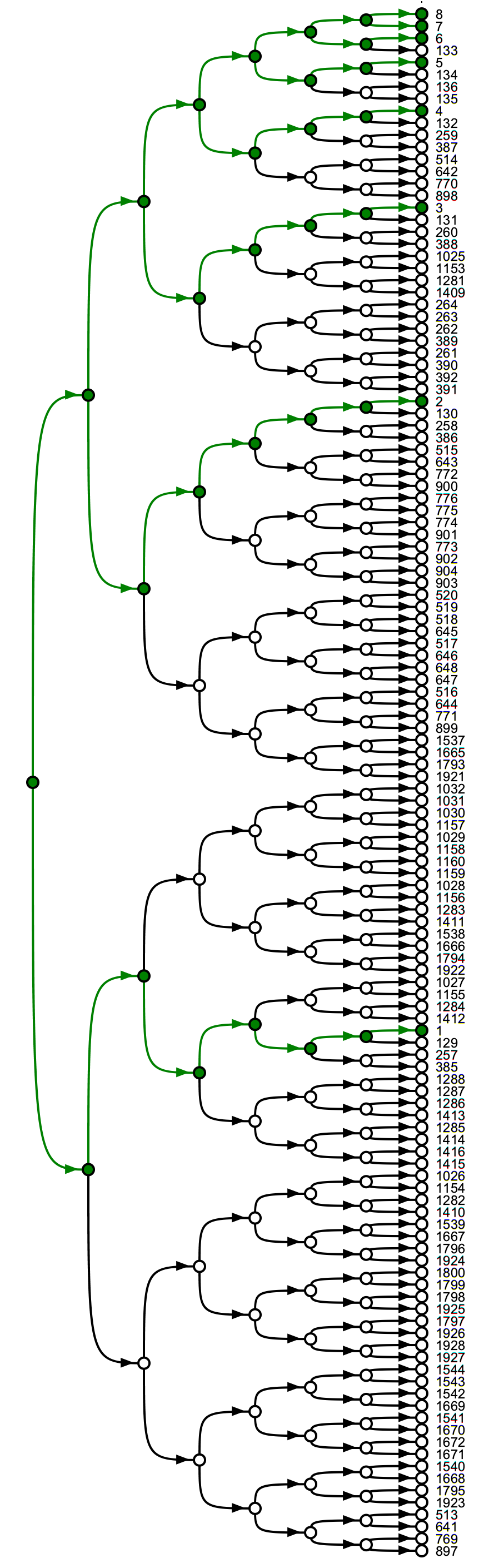}
    \end{minipage}
    \caption{{Subtrees $\cS_1$ (left) and $\cT_1$ (right) that are used to construct a counterexample for Conjecture~\ref{conjecture}. A pair of anti-caterpillars on leaves $\{1,2,\ldots,{8}\}$  is highlighted in blue and green in $\cS_1$ and $\cT_1$, respectively. Each of $\cS_1$ and $\cT_1$ perfectly embeds sixteen 8-caterpillars.}}
    \label{both-subtrees}
\end{figure}

{Now let $\cS$ and $\cT$ be two balanced trees on $2048=2^{4+7}$ leaves. In what follows, we  label the leaves of $\cS$ and $\cT$ such that $\textrm{mast}(\cS,\cT)=32$. Let $\cS_1,\cS_2,\ldots,\cS_{16}$  be the sixteen leaf-disjoint balanced pendant subtrees of $\cS$ that each have size $2^7$, and let $\cT_1,\cT_2,\ldots,\cT_{16}$  be the sixteen leaf-disjoint balanced pendant subtrees of $\cT$ that each have size $2^7$. By Lemma~\ref{label_distribution} and Corollary~\ref{c:caterpillars}, we can bijectively label the leaves of $\cS$ and $\cT$ with the elements in $\{1,2,\ldots,2048\}$ such that the following three properties are satisfied }
\begin{enumerate}[(i)]
\item {for all $i,j\in\{1,2,\ldots,16\}$, the pendant subtrees $\cS_i$ and $\cT_j$ have exactly $|\cL(\cS_i)\cap\cL(\cT_j)|=2^{7-4}=8$ labels in common,}
\item {each balanced pendant subtree in $\{\cS_1,\cS_2,\ldots,\cS_{16}\}$ and $\{\cT_1,\cT_2,\ldots,\cT_{16}\}$ perfectly embeds sixteen 8-caterpillars, and}
\item {the restrictions $\cS_i|(\cL(\cS_i)\cap \cL(\cT_j))$ and $\cT_j|(\cL(\cS_i)\cap \cL(\cT_j))$ are a pair of anti-caterpillars.}
\end{enumerate}
{As a MAST of a pair of anti-caterpillars has size 2 (see Lemma~\ref{lem:packingCaterpillars}), it follows that, if $\cS$ and $\cT$ have their leaves labelled such that (i), (ii), and (iii) hold, then $\textrm{mast}(\cS,\cT)=2\cdot 16=32<2048^{\frac 1 2}$ (see Lemma~\ref{claim3}), which disproves Conjecture~\ref{conjecture}. To illustrate the labelling of $\cS$ and $\cT$, the two subtrees $\cS_1$ and $\cT_1$ with their leaves labelled are shown in Figure~\ref{both-subtrees}. Note that $\cS_1$ and $\cT_1$ embed the pair $(1,2,3,4,5,6,7,8)$ and $(8,7,6,5,4,3,2,1)$ of anti-caterillars as indicated by the edges highlighted in blue and green, respectively.} The full trees in Newick format are given in Appendix~\ref{a:example2048}.

\section{Proof of Theorem~\ref{mainresult}}\label{sec:proofs}

The proof of Theorem~\ref{mainresult} is an amalgamation of Lemmas~\ref{lem:packingCaterpillars}, \ref{label_distribution}, and~\ref{claim3} together with Corollary~\ref{c:caterpillars}. 

\begin{lemma}
    \label{lem:packingCaterpillars}
    Let $\cC$ and $\cC'$ be a pair of anti-caterpillars on at least two leaves. Then {\em mast}$(\cC,\cC')=2$.
\end{lemma}

\begin{proof}
Evidently, $\textrm{mast}(\cC,\cC')\geq 2$.  Let $ab|c$ be a triple that is embedded in $\cC$. Since $\cC$ and $\cC'$ are anti-caterpillars, it follows that either $ac|b$ or $bc|a$ is a triple embedded in $\cC'$, in particular, $ab|c$ is not embedded in $\cC'$. Hence  $\textrm{mast}(\cC,\cC')<3$ and, therefore, $\textrm{mast}(\cC,\cC')=2$.
\end{proof}

\begin{lemma}\label{label_distribution}
Let $h_2\ge h_1$ be non-negative integers. Let $\cS$ (resp.\ $\cT$) be a balanced tree on $2^{h_1+h_2}$ leaves consisting of $2^{h_1}$ balanced pendant subtrees $\cS_1, \cS_2, \ldots, \cS_{2^{h_1}}$ (resp.\ $\cT_1, \cT_2, \ldots, \cT_{2^{h_1}}$) each of size $2^{h_2}$. Then we can bijectively label the leaves of $\cS$ and $\cT$ with the elements in $\left\{1, 2, \ldots, 2^{h_1+h_2}\right\}$ so that, for all $i, j\in \{1, 2, \ldots, 2^{h_1}\}$, the pendant subtrees $\cS_i$ and $\cT_j$ have exactly $2^{h_2-h_1}$ common labels.
\end{lemma}
 
\begin{proof}
For each $r\in\{1, 2, \ldots, 2^{2h_1}\}$, let
$$L_r = \{1+(r-1)\cdot 2^{h_2-h_1},\, 2+(r-1)\cdot 2^{h_2-h_1},\, \ldots,\, 2^{h_2-h_1}+(r-1)\cdot 2^{h_2-h_1}\}.$$
That is, $L_1 = \{1, 2, \ldots, 2^{h_2-h_1}\}$, $L_2 = \{1+2^{h_2-h_1}, 2+2^{h_2-h_1}, \ldots, 2^{h_2-h_1}+ 2^{h_2-h_1}\}$, and so forth. Note that $|L_r|=2^{h_2-h_1}$ for all $r$. Moreover, $L_i \cap L_j = \emptyset$ for all $i\neq j$, and
$$\bigcup_{r=1}^{2^{2h_1}} L_r= \left\{1,2, \ldots, 2^{h_1+h_2}\right\}.$$

Using the sets $L_r$, we now assign labels to the pendant subtrees $\cS_1, \cS_2, \ldots, \cS_{2^{h_1}}$ of $\cS$ and the pendant subtrees $\cT_1, \cT_2, \ldots, \cT_{2^{h_1}}$ of $\cT$ as follows. For all $i$, the subtree $\cS_i$ is assigned the union of label sets in row $i$ and, for all $j$, the subtree $\cT_j$ is assigned the union of label sets in column $j$:

\begin{center}
\begin{tabular}{c|cccccc}
& $\cT_1$ & $\cT_2$ & $\ldots$ & $\cT_j$ & $\ldots$ & $\cT_{2^{h_1}}$ \\
\hline 
$\cS_1$ & $L_1$ & $L_2$ & $\ldots$ & $L_j$ & $\ldots$ & $L_{2^{h_1}}$ \\  
$\cS_2$ & $L_{1+2^{h_1}}$ & $L_{2+2^{h_1}}$ & $\ldots$ & $L_{j+2^{h_1}}$ & $\ldots$ &  $L_{2^{h_1}+2^{h_1}}$ \\ 
$\vdots$ & $\vdots$ & $\vdots$ & $\ldots$ & $\vdots$ & $\ldots$ & $\vdots$ \\ 
$\cS_i$ & $L_{1+(i-1)\cdot 2^{h_1}}$ & $L_{2+(i-1)\cdot 2^{h_1}}$ & $\ldots$ & $L_{j+(i-1) \cdot 2^{h_1}}$ & $\ldots$ & $L_{2^{h_1}+(i-1)\cdot 2^{h_1}}$ \\
$\vdots$ & $\vdots$ & $\vdots$ & $\ldots$ & $\vdots$ \\ 
$\cS_{2^{h_1}}$ & $L_{1+ (2^{h_1}-1) \cdot 2^{h_1}}$ & $L_{2+ (2^{h_1}-1) \cdot 2^{h_1}}$  & $\ldots$ & $L_{j+(2^{h_1}-1) \cdot 2^{h_1}}$ & $\ldots$ & $L_{2^{h_1}+ (2^{h_1}-1) \cdot 2^{h_1}}=L_{2^{2h_1}}$  
\end{tabular} 
\end{center}

\noindent Since $|L_r|=2^{h_2-h_1}$, and $L_i\cap L_j=\emptyset$ for all $i\neq j$, it follows that each subtree $\cS_i$ and $\cT_j$ is assigned $2^{h_1} \cdot 2^{h_2-h_1} = 2^{h_2}$ distinct labels. Moreover, under this assignment of labels, $\cL(\cS_i)\cap \cL(\cT_j)=L_r$ for some $r\in \{1, 2, \ldots, 2^{2h_1}\}$ (namely, the one that is placed in row $i$ and column $j$). In particular,  for all $i$ and $j$, the pendant subtrees $\cS_i$ and $\cT_j$ have exactly $|L_{r}|=2^{h_2-h_1}$ common labels.
\end{proof}

Consider the (infinite) sequence A054243 of integers from {\it The On-Line Encyclopedia of Integer Sequences}~\cite{oeis}. The first eighteen integers of this sequence are:
\begin{center}
\begin{tabular}{ c| c c c c c c c c c c c c c c c c c c }
$n$ &1&2&3&4&5&6&7&8&9&10&11&12&13&14&15&16&17&18 \\\hline 
$a(n)$ &1& 1& 1& 2& 2& 4& 8& 16& 16& 32& 64& 128& 256& 512& 1024& 2048& 2048& 4096
\end{tabular}
\end{center}
and its closed form is $a(n) = 2^{\lfloor n - \log n - 1\rfloor}$. Intuitively, the sequence is formed by successive powers of two, but it ``stutters'' if $n$ is a power of two, that is, if $n$ is a power of two, then $a(n)=a(n+1)$.

\begin{lemma}
\label{l:stutter-seq}
Let $n$ be a positive integer. Then a balanced phylogenetic tree on $2^{n-1}$ leaves embeds
$$2^{\lfloor n-\log n-1\rfloor}$$
label-disjoint $n$-caterpillars.
\end{lemma}

\begin{proof}
The proof is by induction on $n$. Let $\cT$ be a balanced phylogenetic  tree on $2^{n-1}$ leaves. Evidently, if $n=1$, then $\cT$ embeds one $1$-caterpillar, and so the lemma holds. Now suppose that $n\geq 2$, and that the lemma holds for all positive integers at most $n-1$. Let $\cT_1$ and $\cT_2$ denote the two maximal pendant subtrees of $\cT$. For each $i\in \{1, 2\}$, $\cT_i$ is a balanced phylogenetic tree on $2^{n-2}$ leaves and so, by the induction assumption, $\cT_i$ embeds
$$a(n-1)=2^{\lfloor (n-1)-\log (n-1)-1\rfloor}$$
label-disjoint $(n-1)$-caterpillars. The remainder of the proof splits into two cases depending on whether or not $n-1$ is a power of two.

For the first case, assume that $n-1$ is a power of two. Then $a(n)=a(n-1)$, and so we need to show that $\cT$ embeds $a(n-1)$ label-disjoint $n$-caterpillars. If $\cC_{n-1}$ is an $(n-1)$-caterpillar embedded in $\cT_i$, then $\cC_{n-1}||l$ is an $n$-caterpillar embedded in $\cT$, where $l\in \cT_j$ and $\{i, j\}=\{1, 2\}$. Thus, we can extend an embedding $\mathcal E_1$ of $\frac{1}{2}a(n-1)$ label-disjoint $(n-1)$-caterpillars in $\cT_1$ and an embedding $\mathcal E_2$ of $\frac{1}{2}a(n-1)$ label-disjoint $(n-1)$-caterpillars in $\cT_2$ to an embedding of $a(n-1)$ label-disjoint $n$-caterpillars in $\cT$ provided, for each $i\in \{1, 2\}$, we have 
$$2^{n-2}-|\cL(\mathcal E_i)|\geq \textstyle{\frac{1}{2}}a(n-1).$$
That is,
\begin{equation}\label{eq:one}
\textstyle{\frac{1}{2}}\cdot 2^{\lfloor (n-1)-\log (n-1)-1\rfloor}\cdot (n-1)+ \textstyle{\frac{1}{2}}\cdot 2^{\lfloor (n-1)-\log (n-1)-1\rfloor}\leq 2^{n-2}.
\end{equation}
Since $n-1$ is a power of two and $n=2^{\log n}$, the LHS of (\ref{eq:one}) is
\begin{align*}
\textstyle{\frac{1}{2}}\cdot 2^{\lfloor (n-1)-\log (n-1)-1\rfloor}\cdot ((n-1)+1)
& = 2^{-1}\cdot 2^{(n-1)-\log (n-1)-1}\cdot 2^{\log n} \\
& = 2^{n-3-\log (n-1)+\log n}.
\end{align*}
As $-\log (n-1)+\log n\le 1$, we have
$$2^{n-3-\log (n-1)+\log n}\leq 2^{n-3+1}=2^{n-2},$$
thereby establishing (\ref{eq:one}).

For the second case, assume that $n-1$ is not a power of two. Then $a(n)=2a(n-1)$, and so we need to show that $\cT$ embeds $2a(n-1)$ label-disjoint $n$-caterpillars. Analogous to the first case, we can extend an embedding $\mathcal E_1$ of $a(n-1)$ label-disjoint $(n-1)$-caterpillars in $\cT_1$ and an embedding $\mathcal E_2$ of $a(n-1)$ label-disjoint $(n-1)$-caterpillars in $\cT_2$ to an embedding of $2a(n-1)$ label-disjoint $n$-caterpillars in $\cT$ provided, for each $i\in \{1, 2\}$, we have
$$2^{n-2}-|\cL(\mathcal E_i)|\ge a(n-1).$$
That is,
\begin{equation}\label{eq:two}
2^{\lfloor (n-1)-\log (n-1)-1\rfloor}\cdot (n-1) + 2^{\lfloor (n-1)-\log (n-1)-1\rfloor}\le 2^{n-2}.
\end{equation}
Since $n=2^{\log n}$, the LHS of (\ref{eq:two}) is
\begin{align*}
2^{\lfloor (n-1)-\log (n-1)-1\rfloor}\cdot ((n-1)+1)
& = 2^{\lfloor (n-1)-\log (n-1)-1\rfloor}\cdot 2^{\log n} \\
& = 2^{n-2+\lfloor -\log(n-1)\rfloor+\log n}.
\end{align*}
As $n-1$ is not a power of two, it follows that $\lfloor -\log(n-1)\rfloor+\log n\le 0$, and so
$$2^{n-2+\lfloor -\log(n-1)\rfloor+\log n}\le 2^{n-2}.$$
This establishes (\ref{eq:two}), and completes the proof of the lemma.
\end{proof}

\begin{corollary}
\label{c:caterpillars}
Let $k$ be a positive integer. Then a balanced phylogenetic tree $\cT$ of height ${2^k-1}$ perfectly embeds $2^{2^k-k-1}$ caterpillars each of size $2^k$.
\end{corollary}

\begin{proof}
Taking $n=2^k$ in Lemma~\ref{l:stutter-seq}, it follows that $\cT$ embeds
$$2^{\left\lfloor 2^k-\log \left(2^k\right)-1\right\rfloor}=2^{\left\lfloor 2^k-k-1\right\rfloor}=2^{2^k-k-1}$$
label-disjoint $2^k$-caterpillars. Collectively, these caterpillars have
$$2^{2^k-k-1}\cdot 2^{k}=2^{2^k-1}$$
labels and, since the caterpillars are label-disjoint and $\cT$ has $2^{2^k-1}$ leaves, it follows that this embedding is perfect.
\end{proof}
{To illustrate Corollary~\ref{c:caterpillars} for {$k=2$}, a balanced phylogenetic tree of height 3 perfectly embeds two label-disjoint 4-caterpillars. For an example, see Figure~\ref{fig:8caterpillars}, where the two perfectly embedded caterpillars are $(1,2,3,4)$ and $(5,6,7,8)$. Moreover, for {$k=3$}, a balanced phylogenetic tree of height 7 perfectly embeds sixteen label-disjoint 8-caterpillars. Two such trees are shown in Figure~\ref{both-subtrees}, where the pair of anti-caterpillars $(1,2,3,4,5,6,7,8)$ and $(8,7,6,5,4,3,2,1)$ is highlighted in blue and green, respectively.}

\begin{lemma}\label{claim3}
{Let $k,p,q$ and $r$ be positive integers.} Let $\cS$ and $\cT$ be balanced phylogenetic trees on $p\cdot 2^k$ and $q\cdot 2^k$ leaves, respectively. Suppose that $\cS$ consists of $p$ pendant subtrees $\cS_1,\cS_2,\ldots,\cS_{p}$ each of size $2^{k}$ and $\cT$ consists of $q$ pendant subtrees $\cT_1,\cT_2,\ldots,\cT_{q}$ each of size $2^{k}$. If
\begin{enumerate}[{\rm (i)}]
\item for all $i, j$, we have $|\cL(\cS_i)\cap\cL(\cT_j)|=r$ and
\item for all $i, j$, the restrictions $\cS_i|(\cL(\cS_i)\cap \cL(\cT_j))$ and $\cT_j|(\cL(\cS_i)\cap \cL(\cT_j))$ is a pair of anti-caterpillars,
\end{enumerate}  
then, $\textup{mast}(\cS,\cT)\le 2\cdot \max\{p,q\}$.
\end{lemma}

\begin{proof}
Let $f(p,q)$ {denote} the maximum size of a MAST for $\cS$ and $\cT$. We will show by induction on $p+q$ that $f(p,q)\le 2\cdot\max\{p,q\}$. Note that, since $\cS$ and $\cT$ are balanced, $p$ and $q$ are powers of two. Furthermore, by symmetry, $f(p,q)=f(q,p)$. 

Without loss of generality, we may assume $p\leq q$. If $p=1$, then, for all positive integers $q$, we have $f(1,q) = 2q$. To see this, for all $j$, the restricted subtrees $\cS_1|(\cL(\cS_1)\cap \cL(\cT_j))$ and $\cT_j|(\cL(\cS_1)\cap \cL(\cT_j))$ induce a pair of anti-caterpillars, and so, by Lemma~\ref{lem:packingCaterpillars}, each pendant subtree $\cT_j$ of $\cT$ contributes at most $2$ leaves to a MAST between $\cS$ and $\cT$. It follows that the lemma holds for when $p=1$ and $q$ is a positive integer.

Now assume that $p\ge 2$ and the lemma holds for all smaller values of $p+q$. Let $\cS_L, \cS_R$ and $\cT_L, \cT_R$ be the maximal pendant subtrees of $\cS$ and $\cT$, respectively. Let $\cM$ be a MAST of $\cS$ and $\cT$. If there exist $x, y\in \cL(\cM)$ such that $x\in \cL(\cS_L)\cap \cL(\cT_L)$ and $y\in \cL(\cS_L)\cap \cL(\cT_R)$, then $\cL(\cM)\cap \cL(\cS_R)$ is empty. Otherwise, there is a triple $xy|z$ embedded in $\cM$, where $z\in \cL(\cS_R)$, but $xy|z$ is not embedded in $\cT$, a contradiction. Intuitively, if $\cM$ connects one half of $\cS$ to both halves of $\cT$, then the other half of $\cS$ is not used in $\cM$. Using symmetric arguments, we obtain the following:
\begin{enumerate}[(A)]
\item if $\cL(\cM)\cap \cL(\cS_L)\cap \cL(\cT_L)$ and $\cL(\cM)\cap \cL(\cS_L)\cap \cL(\cT_R)$ are non-empty, then $\cL(\cM)\cap \cL(\cS_R)=\emptyset$;

\item if $\cL(\cM)\cap \cL(\cS_R)\cap \cL(\cT_L)$ and $\cL(\cM)\cap \cL(\cS_R)\cap \cL(\cT_R)$ are non-empty, then $\cL(\cM)\cap \cL(\cS_L)=\emptyset$;

\item if $\cL(\cM)\cap \cL(\cS_L)\cap \cL(\cT_L)$ and $\cL(\cM)\cap \cL(\cS_R)\cap \cL(\cT_L)$ are non-empty, then $\cL(\cM)\cap \cL(\cT_R)=\emptyset$; and

\item if $\cL(\cM)\cap \cL(\cS_L)\cap \cL(\cT_R)$ and $\cL(\cM)\cap \cL(\cS_R)\cap \cL(\cT_R)$ are non-empty, then $\cL(\cM)\cap \cL(\cT_L)=\emptyset$.
\end{enumerate}
Thus,
\begin{enumerate}[(I)]
\item either $\cL(\cM)\cap \cL(\cS_L)$ or $\cL(\cM)\cap \cL(\cS_R)$ is empty, or

\item either $\cL(\cM)\cap \cL(\cT_L)$ or $\cL(\cM)\cap \cL(\cT_R)$ is empty, or

\item either
$$(\cL(\cM)\cap \cL(\cS_L)\cap \cL(\cT_R)) \cup(\cL(\cM)\cap \cL(\cS_R)\cap \cL(\cT_L))$$
or
$$(\cL(\cM)\cap \cL(\cS_L)\cap \cL(\cT_L))\cup(\cL(\cM)\cap \cL(\cS_R)\cap \cL(\cT_R))$$
is empty. 
\end{enumerate}
Therefore, $f(p,q)\leq \max \{f(p/2,q),\, f(p,q/2),\, 2f(p/2,q/2)\}.$

If $p\leq q/2$, then, by induction, $f(p, q)\le \max\{2q, q, 2q\}=2q$ and so, as $p\le q$, the lemma holds. If $p>q/2$, then, by induction, $f(p, q)\le \max \{2q, 2p, 2q\}=2q$ since $p\leq q$, and it again follows that the lemma holds. This completes the proof.
\end{proof}

Theorem~\ref{mainresult} is an almost immediate consequence of the next theorem:

\begin{theorem}\label{t:construction}
Let $k$ be a positive integer. Then, there exist balanced phylogenetic $X$-trees $\cS$ and $\cT$, where $|X|=2^{2^{k+1}-k-2}$, such that {\em mast}$(\cS,\cT)=2^{2^k-k}$.
\end{theorem}

\begin{proof}
Let $h_1=2^k-k-1$ and $h_2=2^k-1$. Let $\cS$ and $\cT$ be two (unlabelled) balanced trees of height $h=h_1+h_2$. Thus, $\cS$ and $\cT$ each has size $2^h=2^{2^{k+1}-k-2}$ and, each furthermore, has $2^{h_1}$ balanced pendant subtrees of size $2^{h_2}$. Let $\cS_1,\cS_2,\ldots,\cS_{2^{h_1}}$ denote the pendant subtrees of $\cS$ of size $2^{h_2}$ and let $\cT_1,\cT_2,\ldots,\cT_{2^{h_1}}$ denote the pendant subtrees of $\cT$ of size $2^{h_2}$. 

By Lemma~\ref{label_distribution}, we can bijectively label the leaves of $\cS$ and $\cT$ with the elements in $\left\{1, 2, \ldots, 2^{h_1+h_2}\right\}$ so that, for all $i, j\in \{1, 2, \ldots, 2^{h_1}\}$, the pendant subtrees $\cS_i$ and $\cT_j$ have exactly $2^{h_2-h_1}=2^k$ common labels. Moreover, it follows by Corollary~\ref{c:caterpillars} that, under such a bijection, we can label the leaves of $\cS$ and $\cT$ so that, for all $i, j$, the restrictions $\cS_i|(\cL(\cS_i)\cap \cL(\cT_j))$ and $\cT_j|(\cL(\cS_i)\cap \cL(\cT_j))$ is in fact a pair of anti-caterpillars. With this labelling of $\cS$ and $\cT$, Lemma~\ref{claim3} says that:
$${\rm mast} (\cS, \cT)\le 2\cdot \max\left\{2^{h_1}, 2^{h_1}\right\}=2\cdot 2^{2^k-k-1}=2^{2^k-k}.$$
Since $\cS$ and $\cT$ have the same label sets, it is easily checked that this inequality is an equality, that is,
$${\rm mast}(\cS, \cT)=2^{2^k-k}.$$
\end{proof}

We now prove Theorem~\ref{mainresult}.

\begin{proof}[Proof of Theorem~\ref{mainresult}]
Pick an integer $k$ such that $k > 2\log (1/c)+2$. By Theorem~\ref{t:construction}, there exist balanced phylogenetic trees $\cS$ and $\cT$ on the same label set of size $n=2^{2^{k+1}-k-2}$ such that
$$\textrm{mast}(\cS,\cT)=2^{2^k-k}.$$ 
Since
$$\frac{2^{2^k-k}}{2^{2^k-\frac{k}{2}-1}}=2^{-\frac{k}{2}+1}<c,$$
where $2^{2^k-\frac{k}{2}-1}=n^{\frac 1 2}$, we have
$$2^{2^k-k}= \frac{2^{2^k-k}}{2^{2^k-\frac{k}{2}-1}}\cdot{2^{2^k-\frac{k}{2}-1}} < cn^{\frac 1 2},$$
and the theorem follows. 
\end{proof}

In particular, if we take $c=1$, then following the proof of Theorem~\ref{mainresult}, we choose $k=3$.  It follows that there are two balanced phylogenetic trees on $2^{11}=2048$ leaves such that their MAST has size $2^5=32$, whereas $\sqrt{2048}\approx 45.25$. Applying the construction yields the two trees described in Section~\ref{sec:outline}, and their MAST may be verified to have size 32 by standard computations \cite{phangorn}.

\section{Lower Bounds}
\label{sec:lowerBounds}

The main result of this paper has been to disprove Conjecture~\ref{conjecture} of \cite{martin2013maximum}. 
What then can be said about the size of a MAST of two balanced phylogenetic trees on the same label set? Corollary 16 of the same paper states that two  balanced phylogenetic trees on the same leaf set of cardinality $n$ have a MAST on at least $n^{\beta}$ leaves, where $\beta$ is not explicitly given but is defined by some formulae. Specifically, for any $\delta\in \left(0, \tfrac13-\tfrac{1}{3\sqrt{2}}\right)$, we can take 
$$\beta = \left(\frac{1 + 2 \log (1-3\delta)}{\log(1-3\delta)-\log \delta}\right).$$ 
Numerically, the largest value of $\beta$ that may be inferred from their approach is approximately 0.149.  
We have observed that this lower bound can be slightly improved as follows.

\begingroup
\def\thetheorem{\ref{thm:lowerbound}}
\begin{theorem}
If $\cS$ and $\cT$ are two balanced phylogenetic $X$-trees, where $n=|X|$, then
$${\rm mast}(\cS, \cT)\ge n^{0.17}>n^{\frac{1}{6}}.$$
\end{theorem}
\addtocounter{theorem}{-1}
\endgroup

\begin{proof}
The proof closely follows the approach of~Lemma 15 in~\cite{martin2013maximum} for which Corollary 16 in~\cite{martin2013maximum} is an almost immediate consequence. Let $g(h_1,h_2,t)$ be the minimum size of a MAST over all pairs of balanced phylogenetic trees $\cS$ and $\cT$ such that $\cS$ has $2^{h_1}$ leaves, $\cT$ has $2^{h_2}$ leaves, and the leaf sets overlap in exactly $t>0$ leaves. We first show by induction on $h_1+h_2$ that
\begin{equation}\label{eqn:lowerbound}
    g(h_1,h_2,t)\geq 2^{0.22\log t - 0.025(h_1+h_2)}.
\end{equation}
For the base case, if $h_1+h_2\in \{0,1\}$, then at least one of the trees consists of a single leaf, and so $g(h_1,h_2,1)=1\geq 2^0$ and Equation~\eqref{eqn:lowerbound} holds. Now, assume that
Equation~\eqref{eqn:lowerbound} holds for all pairs of balanced phylogenetic trees whose heights sum to at most $h_1+h_2-1$. Let $\cS_L$ and $\cS_R$ be the two maximal pendant subtrees of $\cS$, and let $\cT_L$ and $\cT_R$ be the two maximal pendant subtrees of $\cT$. By Equation~\eqref{eqn:mast},
\begin{align*}
    \textrm{mast}(\cS, \cT) \geq \max \big\{ & g(h_1-1,h_2-1,|\cL(\cS_L)\cap\cL(\cT_L)|) + g(h_1-1,h_2-1,|\cL(\cS_R)\cap\cL(\cT_R)|), \\
    & g(h_1-1,h_2-1,|\cL(\cS_L)\cap\cL(\cT_R)|) + g(h_1-1,h_2-1,|\cL(\cS_R)\cap\cL(\cT_L)|), \\
    & g(h_1,h_2-1,|\cL(\cS)\cap\cL(\cT_L)|),\, g(h_1,h_2-1,|\cL(\cS)\cap\cL(\cT_R)|), \\
    & g(h_1-1,h_2,|\cL(\cS_L)\cap\cL(\cT)|),\, g(h_1-1,h_2,|\cL(\cS_R)\cap\cL(\cT)|)  \big\}.
\end{align*}
We freely use this inequality in the remainder of the proof.

Without loss of generality, we may assume that the largest overlap between the leaf sets $\cS_L, \cS_R$ and $\cT_L, \cT_R$ is between $\cS_L$ and $ \cT_L$. By the pigeonhole principle, $|\cL(\cS_L)\cap \cL(\cT_L)|\geq t/4$, and so one of the following cases must occur by exhaustion:
\begin{enumerate}[(i)]
    \item $|\cL(\cS_R)\cap \cL(\cT_R)|\geq 0.037t$;\label{case:1}
    \item $|\cL(\cS_L)\cap \cL(\cT_R)|,\, |\cL(\cS_R)\cap \cL(\cT_L)|\geq 0.037t$;\label{case:2}
    \item $|\cL(\cS_R)\cap \cL(\cT_R)|,\, |\cL(\cS_L)\cap \cL(\cT_R)|,\, |\cL(\cS_R)\cap \cL(\cT_L)|< 0.037t$;\label{case:3}
    \item $|\cL(\cS_L)\cap \cL(\cT_R)|\geq0.037t$ and $|\cL(\cS_R)\cap \cL(\cT_R)|,\, |\cL(\cS_R)\cap \cL(\cT_L)|< 0.037t$;\label{case:4}
    \item $|\cL(\cS_R)\cap \cL(\cT_R)|,\, |\cL(\cS_L)\cap \cL(\cT_R)|<0.037t$ and $|\cL(\cS_R)\cap \cL(\cT_L)|\geq 0.037t$.\label{case:5}
\end{enumerate}
In Case (i), it follows by the induction assumption that
\begin{align*}
\textrm{mast}(\cS,\cT)\geq&\ g(h_1-1,h_2-1,0.25t)+g(h_1-1,h_2-1,0.037t)\\
    >&\ 2\cdot 2^{0.22\log 0.037t - 0.025(h_1-1+h_2-1)}\\
    =&\ 2^{1+0.22\log 0.037+0.22\log t + 0.05 - 0.025(h_1+h_2)}\\
    >&\ 2^{0.22\log t - 0.025(h_1+h_2)}.
\end{align*}
In Case (ii), by a similar calculation to that in Case (i), we obtain $$\textrm{mast}(\cS,\cT)\geq 2g(h_1-1,h_2-1,0.037t)>2^{0.22\log t - 0.025(h_1+h_2)}.$$
For Case (iii), since the total overlap in leaf sets is $t$, we must have $|\cL(\cS_L)\cap \cL(\cT_L)|\geq 0.889 t$. So,
\begin{align*}
    \textrm{mast}(\cS,\cT)\geq&\ 2^{0.22\log 0.889t - 0.025(h_1-1+h_2-1)}\\
    =&\ 2^{0.22\log 0.889+0.22\log t + 0.05 - 0.025(h_1+h_2)}\\
    >&\ 2^{0.22\log t - 0.025(h_1+h_2)}.
\end{align*}
Finally, in Case (iv), we have  $|\cL(\cS_L)\cap \cL(\cT)|\geq 0.926t$ and, in Case (v), we have $|\cL(\cS)\cap \cL(\cT_L)|\geq 0.926t$.
In both cases, we have
\begin{align*}
    \textrm{mast}(\cS,\cT)\geq&\ 2^{0.22\log 0.926t - 0.025(h_1+h_2-1)}\\
    =&\ 2^{0.22\log 0.926+0.22\log t + 0.025 - 0.025(h_1+h_2)}\\
    >&\ 2^{0.22\log t - 0.025(h_1+h_2)}.
\end{align*}
Thus, Equation~\eqref{eqn:lowerbound} holds by induction. Substituting $h_1=h_2=\log n$ and $t=n$ into Equation~\eqref{eqn:lowerbound}, we obtain that any two balanced phylogenetic trees on the same leaf set of size $n$ have a MAST on at least
$$g(\log n,\log n,n)\geq 2^{0.22\log n - 0.05\log n} = n^{0.17}>n^{\frac{1}{6}}$$
leaves. This completes the proof of the theorem.
\end{proof}

\section{Acknowledgements}
All authors thank Schloss Dagstuhl---Leibniz Centre for Informatics---for hosting the Seminar 19443 \emph{Algorithms and Complexity in Phylogenetics} in October 2019, where this work was initiated and Prof.~Mike Steel for hosting a workshop in Sumner, New Zealand in March 2020. The first, second and fifth authors thank the New Zealand Marsden Fund for financial support. The first and fourth authors thank the Erskine Visiting Fellowship Programme for supporting their extended visits to the University of Canterbury, New Zealand. The third author was partially supported by the US National Science Foundation (DMS 1847271) and by a grant from the Simons Foundation (355824).  The fourth author thanks the Simons Foundation for their grant to support collaboration and travel. The sixth author thanks the German Academic Scholarship Foundation for a doctoral scholarship, under which parts of this work were conducted. All authors thank Sean Cleary for helpful discussions.

\bibliographystyle{abbrv}
\bibliography{mast}

\appendix
\pagebreak

\section{Example on 2048 leaves}\label{a:example2048}
The following two balanced phylogenetic trees on label set $\{1,2,\ldots,2048\}$ may be shown to have a MAST of size 32 by standard computational means
\cite{phangorn}. They are presented in standard Newick format below.

$\cS=$\texttt{\footnotesize{\seqsplit{(((((((((((1,2),(3,12)),((4,11),(9,10))),(((5,13),(22,30)),((39,47),(55,63)))),((((6,14),(21,29)),((72,80),(88,96))),(((17,18),(19,28)),((20,27),(25,26))))),(((((7,15),(23,31)),((38,46),(53,61))),(((49,50),(51,60)),((52,59),(57,58)))),((((33,34),(35,44)),((36,43),(41,42))),(((37,45),(54,62)),((104,112),(120,128)))))),((((((65,66),(67,76)),((68,75),(73,74))),(((69,77),(86,94)),((103,111),(119,127)))),((((70,78),(85,93)),((8,16),(24,32))),(((81,82),(83,92)),((84,91),(89,90))))),(((((71,79),(87,95)),((102,110),(117,125))),(((113,114),(115,124)),((116,123),(121,122)))),((((97,98),(99,108)),((100,107),(105,106))),(((101,109),(118,126)),((40,48),(56,64))))))),(((((((129,130),(131,140)),((132,139),(137,138))),(((133,141),(150,158)),((167,175),(183,191)))),((((134,142),(149,157)),((200,208),(216,224))),(((145,146),(147,156)),((148,155),(153,154))))),(((((135,143),(151,159)),((166,174),(181,189))),(((177,178),(179,188)),((180,187),(185,186)))),((((161,162),(163,172)),((164,171),(169,170))),(((165,173),(182,190)),((232,240),(248,256)))))),((((((193,194),(195,204)),((196,203),(201,202))),(((197,205),(214,222)),((231,239),(247,255)))),((((198,206),(213,221)),((136,144),(152,160))),(((209,210),(211,220)),((212,219),(217,218))))),(((((199,207),(215,223)),((230,238),(245,253))),(((241,242),(243,252)),((244,251),(249,250)))),((((225,226),(227,236)),((228,235),(233,234))),(((229,237),(246,254)),((168,176),(184,192)))))))),((((((((257,258),(259,268)),((260,267),(265,266))),(((261,269),(278,286)),((295,303),(311,319)))),((((262,270),(277,285)),((328,336),(344,352))),(((273,274),(275,284)),((276,283),(281,282))))),(((((263,271),(279,287)),((294,302),(309,317))),(((305,306),(307,316)),((308,315),(313,314)))),((((289,290),(291,300)),((292,299),(297,298))),(((293,301),(310,318)),((360,368),(376,384)))))),((((((321,322),(323,332)),((324,331),(329,330))),(((325,333),(342,350)),((359,367),(375,383)))),((((326,334),(341,349)),((264,272),(280,288))),(((337,338),(339,348)),((340,347),(345,346))))),(((((327,335),(343,351)),((358,366),(373,381))),(((369,370),(371,380)),((372,379),(377,378)))),((((353,354),(355,364)),((356,363),(361,362))),(((357,365),(374,382)),((296,304),(312,320))))))),(((((((385,386),(387,396)),((388,395),(393,394))),(((389,397),(406,414)),((423,431),(439,447)))),((((390,398),(405,413)),((456,464),(472,480))),(((401,402),(403,412)),((404,411),(409,410))))),(((((391,399),(407,415)),((422,430),(437,445))),(((433,434),(435,444)),((436,443),(441,442)))),((((417,418),(419,428)),((420,427),(425,426))),(((421,429),(438,446)),((488,496),(504,512)))))),((((((449,450),(451,460)),((452,459),(457,458))),(((453,461),(470,478)),((487,495),(503,511)))),((((454,462),(469,477)),((392,400),(408,416))),(((465,466),(467,476)),((468,475),(473,474))))),(((((455,463),(471,479)),((486,494),(501,509))),(((497,498),(499,508)),((500,507),(505,506)))),((((481,482),(483,492)),((484,491),(489,490))),(((485,493),(502,510)),((424,432),(440,448))))))))),(((((((((513,514),(515,524)),((516,523),(521,522))),(((517,525),(534,542)),((551,559),(567,575)))),((((518,526),(533,541)),((584,592),(600,608))),(((529,530),(531,540)),((532,539),(537,538))))),(((((519,527),(535,543)),((550,558),(565,573))),(((561,562),(563,572)),((564,571),(569,570)))),((((545,546),(547,556)),((548,555),(553,554))),(((549,557),(566,574)),((616,624),(632,640)))))),((((((577,578),(579,588)),((580,587),(585,586))),(((581,589),(598,606)),((615,623),(631,639)))),((((582,590),(597,605)),((520,528),(536,544))),(((593,594),(595,604)),((596,603),(601,602))))),(((((583,591),(599,607)),((614,622),(629,637))),(((625,626),(627,636)),((628,635),(633,634)))),((((609,610),(611,620)),((612,619),(617,618))),(((613,621),(630,638)),((552,560),(568,576))))))),(((((((641,642),(643,652)),((644,651),(649,650))),(((645,653),(662,670)),((679,687),(695,703)))),((((646,654),(661,669)),((712,720),(728,736))),(((657,658),(659,668)),((660,667),(665,666))))),(((((647,655),(663,671)),((678,686),(693,701))),(((689,690),(691,700)),((692,699),(697,698)))),((((673,674),(675,684)),((676,683),(681,682))),(((677,685),(694,702)),((744,752),(760,768)))))),((((((705,706),(707,716)),((708,715),(713,714))),(((709,717),(726,734)),((743,751),(759,767)))),((((710,718),(725,733)),((648,656),(664,672))),(((721,722),(723,732)),((724,731),(729,730))))),(((((711,719),(727,735)),((742,750),(757,765))),(((753,754),(755,764)),((756,763),(761,762)))),((((737,738),(739,748)),((740,747),(745,746))),(((741,749),(758,766)),((680,688),(696,704)))))))),((((((((769,770),(771,780)),((772,779),(777,778))),(((773,781),(790,798)),((807,815),(823,831)))),((((774,782),(789,797)),((840,848),(856,864))),(((785,786),(787,796)),((788,795),(793,794))))),(((((775,783),(791,799)),((806,814),(821,829))),(((817,818),(819,828)),((820,827),(825,826)))),((((801,802),(803,812)),((804,811),(809,810))),(((805,813),(822,830)),((872,880),(888,896)))))),((((((833,834),(835,844)),((836,843),(841,842))),(((837,845),(854,862)),((871,879),(887,895)))),((((838,846),(853,861)),((776,784),(792,800))),(((849,850),(851,860)),((852,859),(857,858))))),(((((839,847),(855,863)),((870,878),(885,893))),(((881,882),(883,892)),((884,891),(889,890)))),((((865,866),(867,876)),((868,875),(873,874))),(((869,877),(886,894)),((808,816),(824,832))))))),(((((((897,898),(899,908)),((900,907),(905,906))),(((901,909),(918,926)),((935,943),(951,959)))),((((902,910),(917,925)),((968,976),(984,992))),(((913,914),(915,924)),((916,923),(921,922))))),(((((903,911),(919,927)),((934,942),(949,957))),(((945,946),(947,956)),((948,955),(953,954)))),((((929,930),(931,940)),((932,939),(937,938))),(((933,941),(950,958)),((1000,1008),(1016,1024)))))),((((((961,962),(963,972)),((964,971),(969,970))),(((965,973),(982,990)),((999,1007),(1015,1023)))),((((966,974),(981,989)),((904,912),(920,928))),(((977,978),(979,988)),((980,987),(985,986))))),(((((967,975),(983,991)),((998,1006),(1013,1021))),(((1009,1010),(1011,1020)),((1012,1019),(1017,1018)))),((((993,994),(995,1004)),((996,1003),(1001,1002))),(((997,1005),(1014,1022)),((936,944),(952,960)))))))))),((((((((((1025,1026),(1027,1036)),((1028,1035),(1033,1034))),(((1029,1037),(1046,1054)),((1063,1071),(1079,1087)))),((((1030,1038),(1045,1053)),((1096,1104),(1112,1120))),(((1041,1042),(1043,1052)),((1044,1051),(1049,1050))))),(((((1031,1039),(1047,1055)),((1062,1070),(1077,1085))),(((1073,1074),(1075,1084)),((1076,1083),(1081,1082)))),((((1057,1058),(1059,1068)),((1060,1067),(1065,1066))),(((1061,1069),(1078,1086)),((1128,1136),(1144,1152)))))),((((((1089,1090),(1091,1100)),((1092,1099),(1097,1098))),(((1093,1101),(1110,1118)),((1127,1135),(1143,1151)))),((((1094,1102),(1109,1117)),((1032,1040),(1048,1056))),(((1105,1106),(1107,1116)),((1108,1115),(1113,1114))))),(((((1095,1103),(1111,1119)),((1126,1134),(1141,1149))),(((1137,1138),(1139,1148)),((1140,1147),(1145,1146)))),((((1121,1122),(1123,1132)),((1124,1131),(1129,1130))),(((1125,1133),(1142,1150)),((1064,1072),(1080,1088))))))),(((((((1153,1154),(1155,1164)),((1156,1163),(1161,1162))),(((1157,1165),(1174,1182)),((1191,1199),(1207,1215)))),((((1158,1166),(1173,1181)),((1224,1232),(1240,1248))),(((1169,1170),(1171,1180)),((1172,1179),(1177,1178))))),(((((1159,1167),(1175,1183)),((1190,1198),(1205,1213))),(((1201,1202),(1203,1212)),((1204,1211),(1209,1210)))),((((1185,1186),(1187,1196)),((1188,1195),(1193,1194))),(((1189,1197),(1206,1214)),((1256,1264),(1272,1280)))))),((((((1217,1218),(1219,1228)),((1220,1227),(1225,1226))),(((1221,1229),(1238,1246)),((1255,1263),(1271,1279)))),((((1222,1230),(1237,1245)),((1160,1168),(1176,1184))),(((1233,1234),(1235,1244)),((1236,1243),(1241,1242))))),(((((1223,1231),(1239,1247)),((1254,1262),(1269,1277))),(((1265,1266),(1267,1276)),((1268,1275),(1273,1274)))),((((1249,1250),(1251,1260)),((1252,1259),(1257,1258))),(((1253,1261),(1270,1278)),((1192,1200),(1208,1216)))))))),((((((((1281,1282),(1283,1292)),((1284,1291),(1289,1290))),(((1285,1293),(1302,1310)),((1319,1327),(1335,1343)))),((((1286,1294),(1301,1309)),((1352,1360),(1368,1376))),(((1297,1298),(1299,1308)),((1300,1307),(1305,1306))))),(((((1287,1295),(1303,1311)),((1318,1326),(1333,1341))),(((1329,1330),(1331,1340)),((1332,1339),(1337,1338)))),((((1313,1314),(1315,1324)),((1316,1323),(1321,1322))),(((1317,1325),(1334,1342)),((1384,1392),(1400,1408)))))),((((((1345,1346),(1347,1356)),((1348,1355),(1353,1354))),(((1349,1357),(1366,1374)),((1383,1391),(1399,1407)))),((((1350,1358),(1365,1373)),((1288,1296),(1304,1312))),(((1361,1362),(1363,1372)),((1364,1371),(1369,1370))))),(((((1351,1359),(1367,1375)),((1382,1390),(1397,1405))),(((1393,1394),(1395,1404)),((1396,1403),(1401,1402)))),((((1377,1378),(1379,1388)),((1380,1387),(1385,1386))),(((1381,1389),(1398,1406)),((1320,1328),(1336,1344))))))),(((((((1409,1410),(1411,1420)),((1412,1419),(1417,1418))),(((1413,1421),(1430,1438)),((1447,1455),(1463,1471)))),((((1414,1422),(1429,1437)),((1480,1488),(1496,1504))),(((1425,1426),(1427,1436)),((1428,1435),(1433,1434))))),(((((1415,1423),(1431,1439)),((1446,1454),(1461,1469))),(((1457,1458),(1459,1468)),((1460,1467),(1465,1466)))),((((1441,1442),(1443,1452)),((1444,1451),(1449,1450))),(((1445,1453),(1462,1470)),((1512,1520),(1528,1536)))))),((((((1473,1474),(1475,1484)),((1476,1483),(1481,1482))),(((1477,1485),(1494,1502)),((1511,1519),(1527,1535)))),((((1478,1486),(1493,1501)),((1416,1424),(1432,1440))),(((1489,1490),(1491,1500)),((1492,1499),(1497,1498))))),(((((1479,1487),(1495,1503)),((1510,1518),(1525,1533))),(((1521,1522),(1523,1532)),((1524,1531),(1529,1530)))),((((1505,1506),(1507,1516)),((1508,1515),(1513,1514))),(((1509,1517),(1526,1534)),((1448,1456),(1464,1472))))))))),(((((((((1537,1538),(1539,1548)),((1540,1547),(1545,1546))),(((1541,1549),(1558,1566)),((1575,1583),(1591,1599)))),((((1542,1550),(1557,1565)),((1608,1616),(1624,1632))),(((1553,1554),(1555,1564)),((1556,1563),(1561,1562))))),(((((1543,1551),(1559,1567)),((1574,1582),(1589,1597))),(((1585,1586),(1587,1596)),((1588,1595),(1593,1594)))),((((1569,1570),(1571,1580)),((1572,1579),(1577,1578))),(((1573,1581),(1590,1598)),((1640,1648),(1656,1664)))))),((((((1601,1602),(1603,1612)),((1604,1611),(1609,1610))),(((1605,1613),(1622,1630)),((1639,1647),(1655,1663)))),((((1606,1614),(1621,1629)),((1544,1552),(1560,1568))),(((1617,1618),(1619,1628)),((1620,1627),(1625,1626))))),(((((1607,1615),(1623,1631)),((1638,1646),(1653,1661))),(((1649,1650),(1651,1660)),((1652,1659),(1657,1658)))),((((1633,1634),(1635,1644)),((1636,1643),(1641,1642))),(((1637,1645),(1654,1662)),((1576,1584),(1592,1600))))))),(((((((1665,1666),(1667,1676)),((1668,1675),(1673,1674))),(((1669,1677),(1686,1694)),((1703,1711),(1719,1727)))),((((1670,1678),(1685,1693)),((1736,1744),(1752,1760))),(((1681,1682),(1683,1692)),((1684,1691),(1689,1690))))),(((((1671,1679),(1687,1695)),((1702,1710),(1717,1725))),(((1713,1714),(1715,1724)),((1716,1723),(1721,1722)))),((((1697,1698),(1699,1708)),((1700,1707),(1705,1706))),(((1701,1709),(1718,1726)),((1768,1776),(1784,1792)))))),((((((1729,1730),(1731,1740)),((1732,1739),(1737,1738))),(((1733,1741),(1750,1758)),((1767,1775),(1783,1791)))),((((1734,1742),(1749,1757)),((1672,1680),(1688,1696))),(((1745,1746),(1747,1756)),((1748,1755),(1753,1754))))),(((((1735,1743),(1751,1759)),((1766,1774),(1781,1789))),(((1777,1778),(1779,1788)),((1780,1787),(1785,1786)))),((((1761,1762),(1763,1772)),((1764,1771),(1769,1770))),(((1765,1773),(1782,1790)),((1704,1712),(1720,1728)))))))),((((((((1793,1794),(1795,1804)),((1796,1803),(1801,1802))),(((1797,1805),(1814,1822)),((1831,1839),(1847,1855)))),((((1798,1806),(1813,1821)),((1864,1872),(1880,1888))),(((1809,1810),(1811,1820)),((1812,1819),(1817,1818))))),(((((1799,1807),(1815,1823)),((1830,1838),(1845,1853))),(((1841,1842),(1843,1852)),((1844,1851),(1849,1850)))),((((1825,1826),(1827,1836)),((1828,1835),(1833,1834))),(((1829,1837),(1846,1854)),((1896,1904),(1912,1920)))))),((((((1857,1858),(1859,1868)),((1860,1867),(1865,1866))),(((1861,1869),(1878,1886)),((1895,1903),(1911,1919)))),((((1862,1870),(1877,1885)),((1800,1808),(1816,1824))),(((1873,1874),(1875,1884)),((1876,1883),(1881,1882))))),(((((1863,1871),(1879,1887)),((1894,1902),(1909,1917))),(((1905,1906),(1907,1916)),((1908,1915),(1913,1914)))),((((1889,1890),(1891,1900)),((1892,1899),(1897,1898))),(((1893,1901),(1910,1918)),((1832,1840),(1848,1856))))))),(((((((1921,1922),(1923,1932)),((1924,1931),(1929,1930))),(((1925,1933),(1942,1950)),((1959,1967),(1975,1983)))),((((1926,1934),(1941,1949)),((1992,2000),(2008,2016))),(((1937,1938),(1939,1948)),((1940,1947),(1945,1946))))),(((((1927,1935),(1943,1951)),((1958,1966),(1973,1981))),(((1969,1970),(1971,1980)),((1972,1979),(1977,1978)))),((((1953,1954),(1955,1964)),((1956,1963),(1961,1962))),(((1957,1965),(1974,1982)),((2024,2032),(2040,2048)))))),((((((1985,1986),(1987,1996)),((1988,1995),(1993,1994))),(((1989,1997),(2006,2014)),((2023,2031),(2039,2047)))),((((1990,1998),(2005,2013)),((1928,1936),(1944,1952))),(((2001,2002),(2003,2012)),((2004,2011),(2009,2010))))),(((((1991,1999),(2007,2015)),((2022,2030),(2037,2045))),(((2033,2034),(2035,2044)),((2036,2043),(2041,2042)))),((((2017,2018),(2019,2028)),((2020,2027),(2025,2026))),(((2021,2029),(2038,2046)),((1960,1968),(1976,1984)))))))))));}}}

$\cT=$ \footnotesize{\texttt{\seqsplit{(((((((((((8,7),(6,133)),((5,134),(136,135))),(((4,132),(259,387)),((514,642),(770,898)))),((((3,131),(260,388)),((1025,1153),(1281,1409))),(((264,263),(262,389)),((261,390),(392,391))))),(((((2,130),(258,386)),((515,643),(772,900))),(((776,775),(774,901)),((773,902),(904,903)))),((((520,519),(518,645)),((517,646),(648,647))),(((516,644),(771,899)),((1537,1665),(1793,1921)))))),((((((1032,1031),(1030,1157)),((1029,1158),(1160,1159))),(((1028,1156),(1283,1411)),((1538,1666),(1794,1922)))),((((1027,1155),(1284,1412)),((1,129),(257,385))),(((1288,1287),(1286,1413)),((1285,1414),(1416,1415))))),(((((1026,1154),(1282,1410)),((1539,1667),(1796,1924))),(((1800,1799),(1798,1925)),((1797,1926),(1928,1927)))),((((1544,1543),(1542,1669)),((1541,1670),(1672,1671))),(((1540,1668),(1795,1923)),((513,641),(769,897))))))),(((((((16,15),(14,141)),((13,142),(144,143))),(((12,140),(267,395)),((522,650),(778,906)))),((((11,139),(268,396)),((1033,1161),(1289,1417))),(((272,271),(270,397)),((269,398),(400,399))))),(((((10,138),(266,394)),((523,651),(780,908))),(((784,783),(782,909)),((781,910),(912,911)))),((((528,527),(526,653)),((525,654),(656,655))),(((524,652),(779,907)),((1545,1673),(1801,1929)))))),((((((1040,1039),(1038,1165)),((1037,1166),(1168,1167))),(((1036,1164),(1291,1419)),((1546,1674),(1802,1930)))),((((1035,1163),(1292,1420)),((9,137),(265,393))),(((1296,1295),(1294,1421)),((1293,1422),(1424,1423))))),(((((1034,1162),(1290,1418)),((1547,1675),(1804,1932))),(((1808,1807),(1806,1933)),((1805,1934),(1936,1935)))),((((1552,1551),(1550,1677)),((1549,1678),(1680,1679))),(((1548,1676),(1803,1931)),((521,649),(777,905)))))))),((((((((24,23),(22,149)),((21,150),(152,151))),(((20,148),(275,403)),((530,658),(786,914)))),((((19,147),(276,404)),((1041,1169),(1297,1425))),(((280,279),(278,405)),((277,406),(408,407))))),(((((18,146),(274,402)),((531,659),(788,916))),(((792,791),(790,917)),((789,918),(920,919)))),((((536,535),(534,661)),((533,662),(664,663))),(((532,660),(787,915)),((1553,1681),(1809,1937)))))),((((((1048,1047),(1046,1173)),((1045,1174),(1176,1175))),(((1044,1172),(1299,1427)),((1554,1682),(1810,1938)))),((((1043,1171),(1300,1428)),((17,145),(273,401))),(((1304,1303),(1302,1429)),((1301,1430),(1432,1431))))),(((((1042,1170),(1298,1426)),((1555,1683),(1812,1940))),(((1816,1815),(1814,1941)),((1813,1942),(1944,1943)))),((((1560,1559),(1558,1685)),((1557,1686),(1688,1687))),(((1556,1684),(1811,1939)),((529,657),(785,913))))))),(((((((32,31),(30,157)),((29,158),(160,159))),(((28,156),(283,411)),((538,666),(794,922)))),((((27,155),(284,412)),((1049,1177),(1305,1433))),(((288,287),(286,413)),((285,414),(416,415))))),(((((26,154),(282,410)),((539,667),(796,924))),(((800,799),(798,925)),((797,926),(928,927)))),((((544,543),(542,669)),((541,670),(672,671))),(((540,668),(795,923)),((1561,1689),(1817,1945)))))),((((((1056,1055),(1054,1181)),((1053,1182),(1184,1183))),(((1052,1180),(1307,1435)),((1562,1690),(1818,1946)))),((((1051,1179),(1308,1436)),((25,153),(281,409))),(((1312,1311),(1310,1437)),((1309,1438),(1440,1439))))),(((((1050,1178),(1306,1434)),((1563,1691),(1820,1948))),(((1824,1823),(1822,1949)),((1821,1950),(1952,1951)))),((((1568,1567),(1566,1693)),((1565,1694),(1696,1695))),(((1564,1692),(1819,1947)),((537,665),(793,921))))))))),(((((((((40,39),(38,165)),((37,166),(168,167))),(((36,164),(291,419)),((546,674),(802,930)))),((((35,163),(292,420)),((1057,1185),(1313,1441))),(((296,295),(294,421)),((293,422),(424,423))))),(((((34,162),(290,418)),((547,675),(804,932))),(((808,807),(806,933)),((805,934),(936,935)))),((((552,551),(550,677)),((549,678),(680,679))),(((548,676),(803,931)),((1569,1697),(1825,1953)))))),((((((1064,1063),(1062,1189)),((1061,1190),(1192,1191))),(((1060,1188),(1315,1443)),((1570,1698),(1826,1954)))),((((1059,1187),(1316,1444)),((33,161),(289,417))),(((1320,1319),(1318,1445)),((1317,1446),(1448,1447))))),(((((1058,1186),(1314,1442)),((1571,1699),(1828,1956))),(((1832,1831),(1830,1957)),((1829,1958),(1960,1959)))),((((1576,1575),(1574,1701)),((1573,1702),(1704,1703))),(((1572,1700),(1827,1955)),((545,673),(801,929))))))),(((((((48,47),(46,173)),((45,174),(176,175))),(((44,172),(299,427)),((554,682),(810,938)))),((((43,171),(300,428)),((1065,1193),(1321,1449))),(((304,303),(302,429)),((301,430),(432,431))))),(((((42,170),(298,426)),((555,683),(812,940))),(((816,815),(814,941)),((813,942),(944,943)))),((((560,559),(558,685)),((557,686),(688,687))),(((556,684),(811,939)),((1577,1705),(1833,1961)))))),((((((1072,1071),(1070,1197)),((1069,1198),(1200,1199))),(((1068,1196),(1323,1451)),((1578,1706),(1834,1962)))),((((1067,1195),(1324,1452)),((41,169),(297,425))),(((1328,1327),(1326,1453)),((1325,1454),(1456,1455))))),(((((1066,1194),(1322,1450)),((1579,1707),(1836,1964))),(((1840,1839),(1838,1965)),((1837,1966),(1968,1967)))),((((1584,1583),(1582,1709)),((1581,1710),(1712,1711))),(((1580,1708),(1835,1963)),((553,681),(809,937)))))))),((((((((56,55),(54,181)),((53,182),(184,183))),(((52,180),(307,435)),((562,690),(818,946)))),((((51,179),(308,436)),((1073,1201),(1329,1457))),(((312,311),(310,437)),((309,438),(440,439))))),(((((50,178),(306,434)),((563,691),(820,948))),(((824,823),(822,949)),((821,950),(952,951)))),((((568,567),(566,693)),((565,694),(696,695))),(((564,692),(819,947)),((1585,1713),(1841,1969)))))),((((((1080,1079),(1078,1205)),((1077,1206),(1208,1207))),(((1076,1204),(1331,1459)),((1586,1714),(1842,1970)))),((((1075,1203),(1332,1460)),((49,177),(305,433))),(((1336,1335),(1334,1461)),((1333,1462),(1464,1463))))),(((((1074,1202),(1330,1458)),((1587,1715),(1844,1972))),(((1848,1847),(1846,1973)),((1845,1974),(1976,1975)))),((((1592,1591),(1590,1717)),((1589,1718),(1720,1719))),(((1588,1716),(1843,1971)),((561,689),(817,945))))))),(((((((64,63),(62,189)),((61,190),(192,191))),(((60,188),(315,443)),((570,698),(826,954)))),((((59,187),(316,444)),((1081,1209),(1337,1465))),(((320,319),(318,445)),((317,446),(448,447))))),(((((58,186),(314,442)),((571,699),(828,956))),(((832,831),(830,957)),((829,958),(960,959)))),((((576,575),(574,701)),((573,702),(704,703))),(((572,700),(827,955)),((1593,1721),(1849,1977)))))),((((((1088,1087),(1086,1213)),((1085,1214),(1216,1215))),(((1084,1212),(1339,1467)),((1594,1722),(1850,1978)))),((((1083,1211),(1340,1468)),((57,185),(313,441))),(((1344,1343),(1342,1469)),((1341,1470),(1472,1471))))),(((((1082,1210),(1338,1466)),((1595,1723),(1852,1980))),(((1856,1855),(1854,1981)),((1853,1982),(1984,1983)))),((((1600,1599),(1598,1725)),((1597,1726),(1728,1727))),(((1596,1724),(1851,1979)),((569,697),(825,953)))))))))),((((((((((72,71),(70,197)),((69,198),(200,199))),(((68,196),(323,451)),((578,706),(834,962)))),((((67,195),(324,452)),((1089,1217),(1345,1473))),(((328,327),(326,453)),((325,454),(456,455))))),(((((66,194),(322,450)),((579,707),(836,964))),(((840,839),(838,965)),((837,966),(968,967)))),((((584,583),(582,709)),((581,710),(712,711))),(((580,708),(835,963)),((1601,1729),(1857,1985)))))),((((((1096,1095),(1094,1221)),((1093,1222),(1224,1223))),(((1092,1220),(1347,1475)),((1602,1730),(1858,1986)))),((((1091,1219),(1348,1476)),((65,193),(321,449))),(((1352,1351),(1350,1477)),((1349,1478),(1480,1479))))),(((((1090,1218),(1346,1474)),((1603,1731),(1860,1988))),(((1864,1863),(1862,1989)),((1861,1990),(1992,1991)))),((((1608,1607),(1606,1733)),((1605,1734),(1736,1735))),(((1604,1732),(1859,1987)),((577,705),(833,961))))))),(((((((80,79),(78,205)),((77,206),(208,207))),(((76,204),(331,459)),((586,714),(842,970)))),((((75,203),(332,460)),((1097,1225),(1353,1481))),(((336,335),(334,461)),((333,462),(464,463))))),(((((74,202),(330,458)),((587,715),(844,972))),(((848,847),(846,973)),((845,974),(976,975)))),((((592,591),(590,717)),((589,718),(720,719))),(((588,716),(843,971)),((1609,1737),(1865,1993)))))),((((((1104,1103),(1102,1229)),((1101,1230),(1232,1231))),(((1100,1228),(1355,1483)),((1610,1738),(1866,1994)))),((((1099,1227),(1356,1484)),((73,201),(329,457))),(((1360,1359),(1358,1485)),((1357,1486),(1488,1487))))),(((((1098,1226),(1354,1482)),((1611,1739),(1868,1996))),(((1872,1871),(1870,1997)),((1869,1998),(2000,1999)))),((((1616,1615),(1614,1741)),((1613,1742),(1744,1743))),(((1612,1740),(1867,1995)),((585,713),(841,969)))))))),((((((((88,87),(86,213)),((85,214),(216,215))),(((84,212),(339,467)),((594,722),(850,978)))),((((83,211),(340,468)),((1105,1233),(1361,1489))),(((344,343),(342,469)),((341,470),(472,471))))),(((((82,210),(338,466)),((595,723),(852,980))),(((856,855),(854,981)),((853,982),(984,983)))),((((600,599),(598,725)),((597,726),(728,727))),(((596,724),(851,979)),((1617,1745),(1873,2001)))))),((((((1112,1111),(1110,1237)),((1109,1238),(1240,1239))),(((1108,1236),(1363,1491)),((1618,1746),(1874,2002)))),((((1107,1235),(1364,1492)),((81,209),(337,465))),(((1368,1367),(1366,1493)),((1365,1494),(1496,1495))))),(((((1106,1234),(1362,1490)),((1619,1747),(1876,2004))),(((1880,1879),(1878,2005)),((1877,2006),(2008,2007)))),((((1624,1623),(1622,1749)),((1621,1750),(1752,1751))),(((1620,1748),(1875,2003)),((593,721),(849,977))))))),(((((((96,95),(94,221)),((93,222),(224,223))),(((92,220),(347,475)),((602,730),(858,986)))),((((91,219),(348,476)),((1113,1241),(1369,1497))),(((352,351),(350,477)),((349,478),(480,479))))),(((((90,218),(346,474)),((603,731),(860,988))),(((864,863),(862,989)),((861,990),(992,991)))),((((608,607),(606,733)),((605,734),(736,735))),(((604,732),(859,987)),((1625,1753),(1881,2009)))))),((((((1120,1119),(1118,1245)),((1117,1246),(1248,1247))),(((1116,1244),(1371,1499)),((1626,1754),(1882,2010)))),((((1115,1243),(1372,1500)),((89,217),(345,473))),(((1376,1375),(1374,1501)),((1373,1502),(1504,1503))))),(((((1114,1242),(1370,1498)),((1627,1755),(1884,2012))),(((1888,1887),(1886,2013)),((1885,2014),(2016,2015)))),((((1632,1631),(1630,1757)),((1629,1758),(1760,1759))),(((1628,1756),(1883,2011)),((601,729),(857,985))))))))),(((((((((104,103),(102,229)),((101,230),(232,231))),(((100,228),(355,483)),((610,738),(866,994)))),((((99,227),(356,484)),((1121,1249),(1377,1505))),(((360,359),(358,485)),((357,486),(488,487))))),(((((98,226),(354,482)),((611,739),(868,996))),(((872,871),(870,997)),((869,998),(1000,999)))),((((616,615),(614,741)),((613,742),(744,743))),(((612,740),(867,995)),((1633,1761),(1889,2017)))))),((((((1128,1127),(1126,1253)),((1125,1254),(1256,1255))),(((1124,1252),(1379,1507)),((1634,1762),(1890,2018)))),((((1123,1251),(1380,1508)),((97,225),(353,481))),(((1384,1383),(1382,1509)),((1381,1510),(1512,1511))))),(((((1122,1250),(1378,1506)),((1635,1763),(1892,2020))),(((1896,1895),(1894,2021)),((1893,2022),(2024,2023)))),((((1640,1639),(1638,1765)),((1637,1766),(1768,1767))),(((1636,1764),(1891,2019)),((609,737),(865,993))))))),(((((((112,111),(110,237)),((109,238),(240,239))),(((108,236),(363,491)),((618,746),(874,1002)))),((((107,235),(364,492)),((1129,1257),(1385,1513))),(((368,367),(366,493)),((365,494),(496,495))))),(((((106,234),(362,490)),((619,747),(876,1004))),(((880,879),(878,1005)),((877,1006),(1008,1007)))),((((624,623),(622,749)),((621,750),(752,751))),(((620,748),(875,1003)),((1641,1769),(1897,2025)))))),((((((1136,1135),(1134,1261)),((1133,1262),(1264,1263))),(((1132,1260),(1387,1515)),((1642,1770),(1898,2026)))),((((1131,1259),(1388,1516)),((105,233),(361,489))),(((1392,1391),(1390,1517)),((1389,1518),(1520,1519))))),(((((1130,1258),(1386,1514)),((1643,1771),(1900,2028))),(((1904,1903),(1902,2029)),((1901,2030),(2032,2031)))),((((1648,1647),(1646,1773)),((1645,1774),(1776,1775))),(((1644,1772),(1899,2027)),((617,745),(873,1001)))))))),((((((((120,119),(118,245)),((117,246),(248,247))),(((116,244),(371,499)),((626,754),(882,1010)))),((((115,243),(372,500)),((1137,1265),(1393,1521))),(((376,375),(374,501)),((373,502),(504,503))))),(((((114,242),(370,498)),((627,755),(884,1012))),(((888,887),(886,1013)),((885,1014),(1016,1015)))),((((632,631),(630,757)),((629,758),(760,759))),(((628,756),(883,1011)),((1649,1777),(1905,2033)))))),((((((1144,1143),(1142,1269)),((1141,1270),(1272,1271))),(((1140,1268),(1395,1523)),((1650,1778),(1906,2034)))),((((1139,1267),(1396,1524)),((113,241),(369,497))),(((1400,1399),(1398,1525)),((1397,1526),(1528,1527))))),(((((1138,1266),(1394,1522)),((1651,1779),(1908,2036))),(((1912,1911),(1910,2037)),((1909,2038),(2040,2039)))),((((1656,1655),(1654,1781)),((1653,1782),(1784,1783))),(((1652,1780),(1907,2035)),((625,753),(881,1009))))))),(((((((128,127),(126,253)),((125,254),(256,255))),(((124,252),(379,507)),((634,762),(890,1018)))),((((123,251),(380,508)),((1145,1273),(1401,1529))),(((384,383),(382,509)),((381,510),(512,511))))),(((((122,250),(378,506)),((635,763),(892,1020))),(((896,895),(894,1021)),((893,1022),(1024,1023)))),((((640,639),(638,765)),((637,766),(768,767))),(((636,764),(891,1019)),((1657,1785),(1913,2041)))))),((((((1152,1151),(1150,1277)),((1149,1278),(1280,1279))),(((1148,1276),(1403,1531)),((1658,1786),(1914,2042)))),((((1147,1275),(1404,1532)),((121,249),(377,505))),(((1408,1407),(1406,1533)),((1405,1534),(1536,1535))))),(((((1146,1274),(1402,1530)),((1659,1787),(1916,2044))),(((1920,1919),(1918,2045)),((1917,2046),(2048,2047)))),((((1664,1663),(1662,1789)),((1661,1790),(1792,1791))),(((1660,1788),(1915,2043)),((633,761),(889,1017)))))))))));}}}
\end{document}